\documentclass[11pt]{amsart}
\usepackage{amssymb, amsmath, amsfonts, cite}
\usepackage[colorlinks]{hyperref}
\usepackage[capitalize]{cleveref}
\usepackage[a4paper]{geometry}

\newtheorem{thm}{Theorem}[section]

\newtheorem{prob}[thm]{Problem}
\newtheorem{lem}[thm]{Lemma}
\newtheorem{cor}[thm]{Corollary}

\theoremstyle{definition}
\newtheorem{casethm}{Case}[thm]

\makeatletter \@addtoreset{equation}{section} \makeatother

\def\={\;=\;}
\def\lle{\;\le\;}
\def\gge{\;\ge\;}

\def\Z{\mathbb{Z}}
\def\N{\mathbb{N}}
\def\rmand{\quad\hbox{ and }\quad}
\def\dist{\mathrm{dist}}

\crefname{case}{Case}{Cases}
\crefname{case}{Case}{Cases}
\creflabelformat{case}{~\upshape#2#1#3}
\crefname{casethm}{Case}{Cases}
\crefname{casethm}{Case}{Cases}
\creflabelformat{casethm}{~\upshape#2#1#3}
\crefname{cond}{Condition}{Conditions}
\Crefname{cond}{Condition}{Conditions}
\creflabelformat{cond}{~\upshape(#2#1#3)}
\crefname{def}{Definition}{Definitions}
\Crefname{def}{Definition}{Definitions}
\creflabelformat{def}{~\upshape(#2#1#3)}
\crefname{ineq}{Ineq.}{Ineqs.}
\Crefname{ineq}{Inequality}{Inequalities}
\creflabelformat{ineq}{~\upshape(#2#1#3)}
\crefname{prob}{Problem}{Problems}
\Crefname{prob}{Problem}{Problems}
\creflabelformat{prob}{~\upshape#2#1#3}
\crefname{rl}{Relation}{Relations}
\crefname{rl}{Relation}{Relations}
\creflabelformat{rl}{~\upshape(#2#1#3)}
\crefname{thm}{Theorem}{Theorems}
\Crefname{thm}{Theorem}{Theorems}
\creflabelformat{thm}{~\upshape#2#1#3}

\crefrangeformat{equation}{Eqs.\!~(#3#1#4) ---~(#5#2#6)}

\def\mathllapinternal#1#2{\llap{$\mathsurround=0pt#1{#2}$}}

\def\mathllap{\mathpalette\mathllapinternal}

\allowdisplaybreaks

\usepackage{xcolor}

\begin{document}

\title[A Tutte-type characterization for graph factors]
{A Tutte-type characterization for graph factors}

\author[H. Lu]{Hongliang Lu}
\address{School of Mathematics and Statistics, Xi'an Jiaotong university, 710049 Xi'an, P.\ R.\ China}
\email{luhongliang@mail.xjtu.edu.cn}

\author[D.G.L. Wang]{David G.L. Wang$^{\dag\ddag}$}
\address{
$^\dag$School of Mathematics and Statistics, Beijing Institute of Technology, 102488 Beijing, P.\ R.\ China\\
$^\ddag$Beijing Key Laboratory on MCAACI, Beijing Institute of Technology, 102488 Beijing, P.\ R.\ China}
\email{glw@bit.edu.cn}


\begin{abstract}
Let $G$ be a connected general graph. Let $f\colon V(G)\to \Z^+$ be a function. We show that $G$ satisfies the Tutte-type condition \[ o(G-S)\le f(S)\qquad\text{for all vertex subsets $S$}, \] if and only if it contains a colored $J_f^*$-factor for any $2$-end-coloring, where $J_f^*(v)$ is the union of all odd integers smaller than $f(v)$ and the integer $f(v)$ itself. This is a generalization of the $(1,f)$-odd factor characterization theorem, and answers a problem of Cui and Kano. We also derive an analogous characterization for graphs of odd orders, which addresses a problem of Akiyama and Kano.
\end{abstract}

\maketitle

\section{Introduction}\label{sec:introduction}

This paper concerns Tutte type conditions and the existence of factors in general graphs.
A considerable large number of literatures on graph factors can be found
from Akiyama and Kano's book~\cite{AK11B}, and from Liu and Yu's book \cite{YL09B}.

Tutte's theorem~\cite{Tut47} states that a graph $G$ has a perfect matching if and only if 
\begin{equation}\label{Cond:Tutte}
o(G-S)\lle |S|\qquad\text{for any vertex subset $S$},
\end{equation}
where $o(G-S)$ denotes the number of odd components of the subgraph $G-S$.
Let $H\colon V(G)\to 2^\N$ be a set-valued function.
A spanning subgraph $F$ of $G$ is said to be an $H$-factor if $\deg_F(v)\in H(v)$.
Let $f\colon V(G)\to \Z^+$ be an odd-integer-valued function.
The factor $F$ is said to be a $(1,f)$-odd factor if it is an $H$-factor where $H(v)=\{1,3,5,\ldots,f(v)\}$.
In particular, a perfect matching is called a $1$-factor.

Lov\'asz~\cite{Lov69} proposed the {\em degree prescribed subgraph problem}
of determining the distance of a factor from a given integer set function.
He~\cite{Lov72} considered it with the restriction that the given set function $H$ is allowed,
i.e., that every gap of the set $H(v)$ for each vertex $v$ is at most two.
He also showed that the problem is NP-complete when the function $H$ is not allowed.
Cornu\'ejols~\cite{Cor88} provided a polynomial Edmonds-Johnson type alternating forest algorithm
for the degree prescribed subgraph problem with~$H$ allowed,
which implies a Gallai-Edmonds type structure theorem.

For convenience, we denote 
\[
J_n
\=\begin{cases}
\{1,3,5,\ldots,n\},&\text{if $n$ is odd};\\[4pt]
\{1,3,5,\ldots,n-1,n\},&\text{if $n$ is even}.
\end{cases}
\]
Define $J_f(v)=J_{f(v)}$ for all vertices $v$.
Under this notation, $J_{f}$-factors are exactly $(1,f)$-odd factors when $f(v)$ is odd for each vertex $v$.
Amahashi~\cite{Ama85} gave a Tutte-type characterization for graphs having a ``global odd factor''.
\begin{thm}[Amahashi]\label{thm:Amahashi}
Let $n\ge 2$ be an integer.
A general graph $G$ has a $J_{2n-1}$-factor if and only if
\begin{equation}\label[cond]{cond:Amahashi}
o(G-S)\le(2n-1)|S|\qquad\text{for all vertex subsets $S$}.
\end{equation}
\end{thm}

By ``global odd factor'' we mean that the coefficient $(2n-1)$ in \cref{cond:Amahashi} is an odd integer independent of the vertices.
By generalizing the constant $2n-1$ to an odd-valued function $f(v)$,
Cui and Kano~\cite{CK88} obtained \cref{thm:Amahashi} as follows.
We denote $f(S)=\sum_{v\in S}f(v)$ for any vertex subset $S$.

\begin{thm}[Cui and Kano]\label{thm:CK}
Let $G$ be a general graph of even order. 
Let $f\colon V(G)\to \Z^+$ be a function 
such that $f(v)$ is odd for each vertex $v$.
The graph $G$ has a $J_f$-factor if and only if
\begin{equation}\label[cond]{cond:f}
o(G-S)\le f(S)\qquad\text{for all vertex subsets $S$}.
\end{equation}
\end{thm}

They~\cite{CK88} also proposed the characterization problem for the existence of a ``global even factor''.

\begin{prob}[Cui and Kano]\label{prob:CK}
Is it possible to characterize graphs that satisfy the condition
\begin{equation}\label[cond]{cond:CK:2n}
o(G-S)\le 2n\,|S|\qquad\text{for all vertex subsets $S$},
\end{equation}
in terms of factors?
\end{prob}
In analog with \cref{thm:Amahashi}, 
the present authors~\cite{LW13} obtained the following partial answer to \cref{prob:CK}.

\begin{thm}[Lu and Wang]\label{thm:LW}
Let $n\ge 2$. Let $G$ be a simple connected graph satisfying \cref{cond:CK:2n}.
Then $G$ contains a $J_{2n}$-factor.
\end{thm}

In the spirit of Cui and Kano's generalizing \cref{thm:Amahashi},
Egawa Kano, and Yan \cite{EKY15} generalized \cref{thm:LW}
by allowing the constant $2n$ to vary as a function.

\begin{thm}[Egawa et al.]\label{thm:EKY}
Let $G$ be a simple connected graph.
Let $f\colon V(G)\to\Z^+$ be a function. 
If $G$ satisfies \cref{cond:f}, then $G$ has a $J_f$-factor.
\end{thm}

By setting $S=\emptyset$, we see that \cref{prob:CK} involves only graphs of even order.
Taking account of graphs of odd order,
Akiyama and Kano \cite[Problem 6.14 (2)]{AK11B} presented the following problem in the same manner.
\begin{prob}[Akiyama and Kano]\label{prob:AK}
Let $G$ be a connected simple graph. 
Let $f\colon V(G)\to\Z^+$ be a function such that $f(v)$ is even for each vertex $v$.
If
\begin{equation}\label[cond]{cond:AK}
o(G-S)\lle f(S)\qquad\text{for all nonempty vertex subsets $S$},
\end{equation}
then what factor or property does $G$ have?
\end{prob}

We will characterize graphs of both even and odd orders satisfying the aforementioned Tutte-type conditions,
in terms of the so-called colored factors.
In fact, characterization results for the degree prescribed subgraph problem are rather scanty.
Examples include \cref{thm:CK}, and those on parity interval factors \cite{Lov72}, 
on the general antifactor problem \cite{Sebo93}, 
and on prescriptions whose gaps have the same parity \cite{Sza09}.
In this paper, we provide one more member to the family of characterization results on graph factors, i.e., 
a graph of even order satisfies the Tutte-type condition \eqref{cond:f} if and only if 
it contains a colored $J_f^*$-factor for any $2$-end-coloring; see \cref{thm:main:even}.
It reduces to \cref{thm:CK} by restricting the function $f$ to be odd-integer-valued,
and to \cref{thm:EKY} by coloring all ends in red.
It is also an answer to \cref{prob:CK}.
Together with \cref{thm:main:odd}, which deals with graphs of odd orders,
we obtain an answer to \cref{prob:AK}.

\section{Preliminary}\label{sec:preliminary}

Let $G$ be a general graph allowing both loops and parallel edges,
with vertex set $V(G)$ and edge set $E(G)$.
The idea of coloring each end of every edge is due to Lov\'asz \cite{Lov72}.
In this section, we give an overview of his idea that help interprets negative degrees of vertices,
based on which his structural description works for general graphs.

\subsection{Interpreting negative degrees of vertices}

We say that a general graph is {\em $2$-end-colored} (or with a {\em $2$-end-coloring})
if every end of every its edge is colored in red or in green,
and if the loop ends receive the same color for each loop.
We call a loop with two red ends a red loop, and a loop with two green ends a green loop.
Let $G$ be a general graph with a $2$-end-coloring.
One associates every edge~$e$
a characteristic function $e\colon V(G)\to\{0,\,\pm1,\,\pm2\}$, defined by
\begin{equation}\label[def]{def:cdeg}
e(v)\=\begin{cases}
2,&\text{if $e$ is a red loop and $v$ is the center of $e$};\\
-2,&\text{if $e$ is a green loop and $v$ is the center of $e$};\\
1,&\text{if $e$ is not a loop and $v$ is a red end of $e$};\\
-1,&\text{if $e$ is not a loop and $v$ is a green end of $e$};\\
0,&\text{otherwise, i.e., if $v$ is not incident with $e$}.
\end{cases}
\end{equation}
Define the {\em colored degree} of a vertex $v$ by 
\[
\Phi_G(v)\=\sum_{e\in E(G)}e(v).
\]
Alternatively, one may consider every red end is weighted by $+1$,
and every green end is weighted by $-1$. In this way, 
the colored degree of a vertex $v$ is the weight sum of ends that incident with $v$.

\subsection{The structural description}

A set $\{h_1,h_2,\ldots,h_m\}$ of increasing integers is said to be {\em allowed}
if $h_{i+1}-h_i\le2$ for all $1\le i\le m-1$.
Let $H\colon V(G)\to 2^\Z$ be an allowed set function, that is, the set $H(v)$ is allowed for each vertex $v$.
A factor~$F$ of the graph~$G$ is said to be a {\em colored $H$-factor}
if $\Phi_F(v)\in H(v)$ for each vertex~$v$.
The distance of the colored degree~$\Phi_F(v)$ from the set $H(v)$ is defined by
\[
\dist_F(v,\,H(v))\=\min\bigl\{|\Phi_F(v)-h|\,\colon\,h\in H(v)\bigr\}.
\]
Lov\'asz \cite{Lov72} introduced the functions
\begin{align}
\delta_H(F)&\=\sum\limits{_{v\in V(G)}}\dist_F(v,\,H(v)),\rmand\label[def]{def:deltaH}\\
\delta(H)&\=\min\{\delta_H(F)\,\colon\,\mbox{$F$ is a subgraph of~$G$}\}.\label[def]{def:delta}
\end{align}
The factor~$F$ is said to be $H$-optimal if $\delta_H(F)=\delta(H)$.
It is an $H$-factor if and only if $\delta_H(F)=0$.
Denote
\begin{equation}\label[def]{def:I}
I_H (v)\=\{\Phi_F(v)\,\colon\,\hbox{$F$ is an $H$-optimal subgraph}\}.
\end{equation}
The vertex set $V(G)$ can be decomposed as
$V(G)=A_H\sqcup B_H\sqcup C_H\sqcup D_H$,
where
\begin{align*}
C_H&\=\{v\in V(G)\,\colon\,I_H(v)\subseteq H(v)\},\\[5pt]
A_H&\=\{v\in V(G)\backslash C_H\,\colon\,\min I_H(v)\ge \max H(v)\},\\[5pt]
B_H&\=\{v\in V(G)\backslash C_H\,\colon\,\max I_H(v)\le \min H(v)\},\rmand\\[5pt]
D_H&\=V(G)\backslash A_H\backslash B_H\backslash C_H.
\end{align*}
The $4$-tuple $(A_H,B_H,C_H,D_H)$ is said to be the {\em $H$-decomposition} of~$G$.
In~\cite[Corollary (2.4)]{Lov72},
Lov\'asz gave the next result.

\begin{lem}[Lov\'asz]\label{lem:CD}
The graph $G$ has no edge between the vertex subsets $C_H$ and $D_H$.
\end{lem}

Let $X$ and $Y$ be disjoint vertex subsets of the graph $G$.
Denote by $E(X,Y)$ the set of edges with one end in~$X$ and the other end in~$Y$.
Define the set function 
\hbox{$H_{X,Y}\colon V(G)-X-Y \to 2^\Z$}
by
\[
H_{X,Y}(z)\=H(z)-\sum_{e\colon\,e(Y)-e(X)=1}e(z).
\]
In particular, we denote $H_X=H_{X,\emptyset}$. On the other hand, define
\[
\nu(X,Y)\=\sum_{e\colon e(Y)-e(X)\geq 1}|e(Y)-e(X)|
\rmand
\nu(X)\=\nu(X,\emptyset).
\]
Note that for any edge $e$,
a vertex in the set $X$ having a non-trivial contribution to the number $e(X)$ must be an end of $e$.
It follows that $e(X)\in\{0,\,\pm 1,\,\pm2\}$ for any edge $e$ and for any vertex subset $X$.
Therefore, we have 
\begin{align}
\nu(X)&\=\sum_{-e(X)\geq 1}|e(X)|\notag\\
&\=|\{e\in E(G)\colon e(X)=-1\}|+2|\{e\in E(G)\colon e(X)=-2\}|.\label{nu}
\end{align}
For any vertex subset $S$,
we denote by~$G[S]$ the subgraph induced by~$S$.
Denote the number of components of~$G[S]$ by~$c(S)$.
In~\cite[Theorem (4.3)]{Lov72}, Lov\'asz established a formula for the number $\delta(H)$.
\begin{thm}[Lov\'asz]\label{thm:deltaH}
We have
\[
\delta(H)
\=c(D_H)+\sum_{v\in B_H}\min H(v)-
\sum_{v\in A_H}\max H(v)-\nu(A_H,B_H).
\]
\end{thm}

For any set $Y$ of integers,
we denote its convex hull by $[Y]=\{y\colon \min Y\le y\le \max Y\}$.
In~\cite[Theorem~(2.1)]{Lov72},
Lov\'asz gave the following property for the vertex subset $D_H$.

\begin{lem}[Lov\'asz]\label{lem:interval}
Suppose that $D_H\neq\emptyset$. Let $v\in D_H$.
Then the set $I_H(v)$ is an allowed set.
Moreover, we have
\begin{itemize}
\item[(i)]
if $\{u\pm 1\}\subseteq I_H(v)$ and $u\not\in I_H(v)$, then $u\in H(v)$ and $u\pm 1\not\in H(v)$;
\smallskip
\item[(ii)]
neither the intersection $[I_H(v)]\cap H(v)$ nor the difference $[I_H(v)]\backslash H(v)$
contains a pair of consecutive integers.
\end{itemize}
\end{lem}

The graph $G$ is said to be {\em $H$-critical} if it is connected and $D_H=V(G)$.
Lov\'asz~\cite[Lemma~(4.1)]{Lov72} showed the following property.
\begin{lem}[Lov\'asz]\label{lem:critical}
If $G$ is an $H$-critical graph, then $\delta(H)=1$.
\end{lem}

In \cite[Theorem~(4.2)]{Lov72},
he also showed that any component of the subgraph $G[D_H]$ is $H_{A_H,B_H}$-critical.
In view of \cref{lem:critical}, this can be stated as follows.

\begin{lem}[Lov\'asz]\label{lem_delta=1}
Suppose that $D_H\neq\emptyset$.
Then for any $H'$-optimal subgraph~$F$ of any component of the subgraph~$G[D_H]$,
we have $\delta_{H'}(F)=1$, where $H'=H_{A_H,B_H}$.
\end{lem}

\section{Main result}\label{sec:main}

This section devotes to the main results of this paper.

Let $G$ be a graph with two vertex subsets $S$ and $T$. 
We denote by $E(S,T)$ the set of edges $e$ with
one end in $S$ and the other end in $T$. When $T=V(G)\setminus S$, 
we use the notation $\partial(S)=\partial (T)=E(S,T)$.
Let $f\colon V(G)\to \Z^+$ be a function. Define
\begin{align}
J_f^*(v)&\=J_f(v)\cup \{-1,\,-3,\,-5,\,\ldots\}\notag\\
&\=\begin{cases}
\{\ldots,\,-5,\,-3,\,-1,\,1,\,3,\,\ldots,\,f(v)\},&\text{if $f(v)$ is odd}\\[4pt]
\{\ldots,\,-5,\,-3,\,-1,\,1,\,3,\,\ldots,\,f(v)-1,\,f(v)\},&\text{if $f(v)$ is even}
\end{cases}\label[def]{def:Jf*}
\end{align}
for each vertex $v$ of $G$.
From definition, we see that the set function $J_f^*$ is allowed.
Throughout this section, 
we let $(A,B,C,D)$ be the $J_f^*$-decomposition of $G$.
Denote $I=I_{J_f^*}$; see \cref{def:I}.

\Cref{thm:main:even,thm:main:odd} treat graphs of even and odd orders respectively.

\begin{thm}\label{thm:main:even}
Let $G$ be a connected general graph.
Let $f\colon V(G)\to \Z^+$ be a function.
Then the graph $G$ satisfies \cref{cond:f}, i.e.,
\[
o(G-S)\lle f(S)\qquad\text{for all vertex subsets $S$}.
\]
if and only if it contains a colored $J_f^*$-factor for any $2$-end-coloring.
\end{thm}

\begin{proof}
{\bf Necessity}.
Suppose that $G$ contains a vertex subset $S$ such that 
\begin{equation}\label[ineq]{pf11}
o(G-S)\gge f(S)+1.
\end{equation}
We shall show that there exists a $2$-end-coloring for which $G$ has no $J_f^*$-factors.
We consider the $2$-end-coloring defined by that an end is colored in red 
if and only if its incident vertex belongs to the set $S$.
Let $F$ be a $J_f^*$-factor of $G$.

Assume that $E(F)\cap (\partial_F C)=\emptyset$ for some odd component of the subgraph $G-S$.
By parity argument, the component $C$ contains a vertex $v$ having an even degree in $F$.
From definition, every edge having $v$ as an end contributes 
the weight $-1$ to the colored degree of $v$. Thus the colored degree of $v$ in the factor $F$ is a negative even integer,
contradicting \cref{def:Jf*} of the function $J_f^*$.

Otherwise, we have $E(F)\cap (\partial C)\ne\emptyset$ for every odd component $C$ of $G-S$.
It follows that $|\partial_F S|\ge o(G-S)\ge f(S)+1$ by \cref{pf11}.
Consequently, there exists a vertex $u\in S$ such that $|\partial_F(u)|\ge f(u)+1$.
Since the colored degree of $u$ in the factor $F$ is at least $|\partial_F(u)|$,
we infer that the vertex~$u$ has colored degree larger than $f(v)$ in $F$,
contradicting \cref{def:Jf*} of the function $J_f^*$ again.

\medskip
\noindent{\bf Sufficiency}.
By way of contradiction,
let $G$ be a graph satisfying \cref{cond:f} without colored $J_f^*$-factors.
From \cref{def:delta}, we have
\begin{equation}\label[ineq]{deltaH>0:even}
\delta(J_f^*)\;>\;0.
\end{equation}

Assume that $B\ne\emptyset$.
Let $v_B\in B$.
From definition, we see that
\[
\Phi_F(v_B)\lle \max I(v_B)\lle \min J_f^*(v_B)
\]
for any $J_f^*$-optimal subgraph~$F$, contradicting the definition of the function~$J_f^*$.
Therefore, we have
\begin{equation}\label{B=0:even}
B\=\emptyset.
\end{equation}
As a consequence, \cref{thm:deltaH,deltaH>0:even} imply that
\begin{equation}\label[ineq]{cD>2A:even}
c(D)\;>\;\sum_{v_A\in A}\max J_f^*(v_A)+\nu(A)\=f(A)+\nu(A).
\end{equation}

From \cref{nu}, we can infer that 
\[
|\{e\in E(A,D)\colon e(A)=-1\}|\lle|\{e\in E(G)\colon e(A)=-1\}|\lle \nu(A).
\]
Consequently,
there are at most $\nu(A)$ components $T$ of the subgraph $G[D]$ such that 
\[
\{e\in E(A,T)\colon e(A)=-1\}\;\ne\;\emptyset.
\]
In other words, the subgraph $G[D]$ has at least $c(D)-\nu(A)$ components 
such that all edges connecting these components with $A$ have red ends in $A$.
Together with \cref{cD>2A:even},
we can suppose that the subgraph $G[D]$ has $q$ components $D_1$, $D_2$, $\ldots$, $D_q$ with $q>f(A)$,
such that 
\[
e(v_A)\ge 0\qquad\text{for each vertex $v_A\in A$ and for all edges $e\in \partial (D')$}, 
\]
where $D'=\cup_{i=1}^qD_i$.
In particular, we find $D'\neq \emptyset$. 
Let $v\in D'$. From definition, we have
\begin{equation}\label{HA=A:even}
(J_f^*)_{A}(v)
\=J_f^*(v)-\sum_{e\colon e(A)=-1}e(v)
\=J_f^*(v).
\end{equation}
By \cref{lem_delta=1,HA=A:even},
we have $\delta_{J_f^*}(F_i)=1$
for any $J_f^*$-optimal subgraph~$F_i$ of any component $D_i$, where $i\in[q]$.
In particular, there exists a unique vertex $v_0\in V(D_i)$ such that
\begin{equation}\label[rl]{v0:even}
\Phi_{F_i}(v_0)\;\not\in\; J_f^*(v_0).
\end{equation}

Since $v\in D'$, we have $v\not\in A\cup C$. From definition, we have
\begin{equation}\label[ineq]{minI:even}
\min I(v)\lle \max J_f^*(v)-1\=f(v)-1.
\end{equation}
We claim that 
\begin{equation}\label[ineq]{maxI:even}
\max I(v)
\lle\begin{cases}
f(v)-1,&\text{if $f(v)$ is even};\\[5pt]
f(v)+1,&\text{if $f(v)$ is odd}.
\end{cases} 
\end{equation}
In fact, \cref{maxI:even} can be shown by \cref{lem:interval,minI:even}.
We handle it in two cases according to the parity of $f(v)$.
 
\begin{casethm}\label{casethm:feven} $f(v)$ is even.

Assume that $f(v)\in I(v)$.
From \cref{lem:interval} (ii), we infer that $f(v)-1\not\in I(v)$.
Then, from \cref{lem:interval} (i), we deduce $f(v)-2\not\in I(v)$.
In view of \cref{minI:even}, we derive that the set $I(v)$ is not allowed, contradicting \cref{lem:interval}.
Therefore, we have $f(v)\not\in I(v)$.
Now, assume $\max I(v)\ge f(v)+1$.
Since the set $I(v)$ is allowed, in view of \cref{minI:even}, we infer that $\{f(v)\pm1\}\subseteq I(v)$.
From \cref{lem:interval} (i), we deduce that $f(v)-1\not\in J_f^*(v)$, a contradiction.
Therefore, we have 
$\max I(v)\le f(v)-1$.
\end{casethm}

\begin{casethm}\label{casethm:fodd} $f(v)$ is odd.

Suppose that there exists $r>f(v)+1$ such that $r\in I(v)$. Then there exists a $J_f^*$-optimal subgraph $F$ such that $\Phi_F(v)=r$.
Since $r\ge 3$, there exists an edge $e$ such that $e(v)>0$, that is, $e(v)\in\{1,2\}$.
Consider the subgraph $F-e$.

If $e(v)=2$, then the edge $e$ is a loop with two red ends.
In this case, the distance 
\[
\dist_{F-e}(v,\,J_f^*(v))
\=r-2-f(v)
\;<\;\dist_{F}(v,\,J_f^*(v))-2.
\]
In view of \cref{def:deltaH}, it follows that $\delta_{J_f^*}(F-e)<\delta_{J_f^*}(F)$,
contradicting the optimality of $F$.
Otherwise, we have $e(v)=1$. 
Then 
\[
\dist_{F-e}(v,\,J_f^*(v))
\=r-1-f(v)
\=\dist_{F}(v,\,J_f^*(v))-1.
\]
Let $e=uv$. Since 
\[
\dist_{F-e}(u,\,J_f^*(u))
\lle\dist_{F}(u,\,J_f^*(u))+1,
\]
we infer that $\delta_{J_f^*}(F-e)\le \delta_{J_f^*}(F)$ from \cref{def:deltaH}.
Since $F$ is optimal, namely $\delta_{J_f^*}(F-e)\ge \delta_{J_f^*}(F)$,
we deduce that the subgraph $F-e$ is optimal.
From \cref{def:I} of $I(v)$, we find $r-1\in I(v)$.
Now, we have $\{r-1,\,r\}\subseteq I(v)\setminus J_f^*(v)$,
contradicting \cref{lem:interval} (ii).
This proves \cref{maxI:even}. 
\end{casethm}

We shall show that the cardinalities $|D_i|$ are odd. 

For any $J_f$-optimal subgraph $F_i$, we have for $v\in V(D_i)$,
\begin{equation}\label[ineq]{Phi<=1}
\Phi_{F_i}(v)
\lle\begin{cases}
f(v)-1,&\text{if $f(v)$ is even};\\[5pt]
f(v)+1,&\text{if $f(v)$ is odd}.
\end{cases} 
\end{equation}
Recall from \cref{v0:even} that $\Phi_{F_i}(v)\;\not\in\; J_f^*(v)$ if and only if $v=v_0$.
We observe that the colored degree~$\Phi_{F_i}(v)$ is even if and only if $v=v_0$.
On the other hand, we claim that the colored degree sum
\[
\sum_{\mathllap{v\in V(D_i)}}\Phi_{F_i}(v)
\=\sum_{\text{$e=vv'\in E(F_i)$ is not a loop}}\!\!\!\!\!\!\!\!\!\!\!\!\!\!\!\!\bigl(e(v)+e(v')\bigr)
+\sum_{\text{$e\in E(F_i)$ is a loop centered at $v$}}\!\!\!\!\!\!\!\!\!\!\!\!\!\!\!\!e(v)
\]
is even. In fact, from \cref{def:cdeg}, 
the summand $e(v)+e(v')\in\{0,\pm2\}$ when the edge $e=vv'$ is not a loop,
and the summand $e(v)\in\{\pm2\}$ when the edge $e$ is a loop centered at $v$.
It follows that the component~$D_i$ is of odd order.
From \cref{lem:CD,cD>2A:even,B=0:even}, we derive that
\begin{equation}\label[ineq]{pf33}
o(G-A)\gge q\;>\;f(A),
\end{equation}
contradicting \cref{cond:f}. 
This completes the proof.
\end{proof}

It is easy to observe that when the integer $f(v)$ is odd for each vertex $v$, 
the graph $G$ contains a $(1,f)$-odd-factor if and only if it contains a colored $J_f^*$-factor
for any $2$-end-coloring. In this sense, \cref{thm:main:even} is a generalization of \cref{thm:CK}.
On the other hand, by coloring all ends in red,
the sufficiency part of \cref{thm:main:even} reduces to Egawa et al.'s \cref{thm:EKY}.

By taking $S=\emptyset$, we find \cref{cond:f} implies the evenness of the order $|G|$.
For graphs of odd orders, 
a slightest remedy to \cref{cond:f} might be removing the requirement for empty sets $S$.

\begin{thm}\label{thm:main:odd}
Let $G$ be a connected general graph of odd order.
Let $f\colon V(G)\to \Z^+$ be a function.
Then the graph $G$ satisfies the Tutte-type condition
\begin{equation}\label[cond]{cond:f:odd}
o(G-S)\le f(S)\qquad\text{for all non-empty vertex subsets $S$},
\end{equation}
if and only if for every $2$-end-coloring, either $G$ is $J_f^*$-critical, or $G$ contains a colored $J_f^*$-factor.
\end{thm}

\begin{proof}
{\bf Necessity}.
Suppose that $G$ contains a nonempty vertex subset $S$ satisfying \cref{pf11}.
Same to the necessity part of the proof of \cref{thm:main:even},
we can show that the graph $G$ admits a $2$-end-coloring for which $G$ has no $J_f^*$-factors.
It suffices to show that $G$ is not $J_f^*$-critical for the same coloring.
Assuming that the graph $G$ is $J_f^*$-critical, namely, $D=V(G)$,
we shall show that the empty set $A\cup C$ contains the nonempty set $S$, which is absurd.
In other words, we will prove that 
\begin{equation}\label[ineq]{dsr1}
\Phi_F(v)\ge f(v)
\qquad\text{for any $J_f^*$-optimal subgraph $F$ and for any vertex $v\in S$}.
\end{equation}
Let $F$ be a $J_f^*$-optimal subgraph.
From \cref{lem:critical}, we infer that 
\begin{equation}\label{pf22}
\delta_{J_f^*}(F)=1\qquad\text{for any $J_f^*$-optimal subgraph $F$}.
\end{equation}
By parity argument, we have $E(S,T)\ne \emptyset$ for any odd component $T$ of the subgraph $G-S$ such that $v_0\not\in T$.

Assume that $v_0\in S$. 
By \cref{pf11}, we have 
\[
\sum_{v\in S}\Phi_{F}(v)
\gge o(G-S)
\gge f(S)+1.
\]
If there is a vertex $v\in S$ and a  such that $\Phi_F(v)<f(v)$,
then either there are two other vertices $v_1,v_2\in S$ such that $\Phi_F(v_i)\ge f(v_i)+1$ for $i=1,2$,
or there is a vertex $v_3\in S$ such that $\Phi_F(v_3)\ge f(v_3)+2$.
In either case, we have $\delta_{J_f^*}(F)\ge 2$ taking account of the contributions of the vertices $v_i$ to
the total distance between the optimal subgraph $F$ and the function $J_f^*$,
contradicting \cref{pf22}.

Otherwise, we have $v_0\not\in S$. In this case, there are at least $o(G-S)-1$ odd components $T$ of the subgraph $G-S$ 
such that $|E_F(S,T)|\ge 1$. Therefore, we infer that 
\[
\sum_{v\in S}\Phi_F(v)
\gge o(G-S)-1
\gge f(S).
\]
Since $v_0\not\in S$, we have $\Phi_F(v)\in J_f^*(v)$ for any vertex $v\in S$.
Hence we deduce that $\Phi_F(v)=f(v)$ for all vertices $v\in S$.
This guarantees \cref{dsr1}, and completes the proof of the sufficiency part.

\medskip
\noindent{\bf Sufficiency}.
By way of contradiction,
let $G$ be a graph satisfying \cref{cond:f:odd} without colored $J_f^*$-factors, and is not $J_f^*$-critical.
Suppose that $V(G)=C\cup D$. By \cref{lem:CD} and the connectivity of $G$, we have $V(G)=C$ or $V(G)=D$.
In the former case, the graph $G$ has a $J_f^*$-factor, while in the latter case, the graph $G$ is $J_f^*$-critical.
Same to the sufficiency part of the proof of \cref{thm:main:even}, 
we have $B=\emptyset$. Thus we find $A\ne \emptyset$.
Same to the remaining proof of \cref{thm:main:even},
we have $o(G-A)>f(A)$, contradicting \cref{cond:f:odd}.
This completes the proof.
\end{proof}

As an application of \cref{thm:main:odd}, we have the following corollary.

\begin{cor}
Let $G$ be a connected general graph of odd order, with a $2$-end coloring.
Let $f\colon V(G)\to \Z^+$ be a function such that $f(v)$ is even for each vertex $v$.
Suppose that \cref{cond:f:odd} holds true.
If the colored degree of each vertex $v$ is at least $f(v)$,
then $G$ contains a colored $J_f^*$-factor.
\end{cor}

\begin{proof}
By \cref{thm:main:odd}, it suffices to show that $G$ is not $J_f^*$-critical.

Suppose that $G$ is $J_f^*$-critical. 
Let $F$ be a $J_f^*$-optimal subgraph with a maximal edge set.
By \cref{lem:critical}, there is a vertex $v_0$ such that $\Phi_F(v_0)\not\in J_f^*(v_0)$.

Assume that $\Phi_F(v_0)\le f(v_0)-2$.
From premise, we see that $\Phi_G(v_0)\ge f(v_0)$.
Thus there is an edge $e=v_0u_0$ such that $e(v_0)\ge 1$.
If the edge $e$ is a loop, i.e., $u_0=v_0$, then 
the distance between $\Phi_{F\cup e}(v_0)$ and the function $J_f^*(v_0)$ is 
either equal to or one less than the distance between $\Phi_{F}(v_0)$ and $J_f^*(v_0)$, namely,
\[
\dist_{F\cup e}(v_0,\,J_f^*(v_0))
\in
\{0,1\}.
\]
Since the subgraph $F$ is optimal, it is impossible that $\dist_{F\cup e}(v_0,\,J_f^*(v_0))=0$.
In other words, the subgraph $F\cup e$ must be optimal, contradicting the choice of $F$.
When $e$ is not a loop, the colored degree $\Phi_{F\cup e}(v_0)$ must be in the set $J_f^*(v_0)$.
Since 
\begin{equation}\label[ineq]{ineq:u0}
|\dist_{F\cup e}(u_0,\,J_f^*(u_0))-\dist_{F}(u_0,\,J_f^*(u_0))|\le 1,
\end{equation}
we infer that the subgraph $F\cup e$ must be optimal, the same contradiction.

Below we can suppose that $\Phi_F(v_0)\ge f(v_0)+1$. From definition, we have
\[
f(v_0)+1\in I(v_0).
\]
Since $\delta(F)=1$, we find $\Phi_F(v_0)=f(v_0)+1>0$.
Pick up an edge $e\in E(F)$ such that $e(v_0)\ge 1$.
If $e$ is a loop, then $\Phi_{F-e}(v_0)=f(v_0)-1\in J_f^*(v_0)$.
Thus the subgraph $F-e$ is a $J_f^*$-factor, a contradiction. 
Otherwise, we can suppose that $e=u_0v_0$ with $u_0\ne v_0$.
In this case, we have $\Phi_{F-e}(v_0)=f(v_0)\in J_f^*(v_0)$.
By \cref{ineq:u0}, we infer that the subgraph $F-e$ is also $J_f^*$-optimal.
Thus we have
\[
f(v_0)\in I(v_0).
\]
Since $\{f(v_0),\,f(v_0)-1\}\subset J_f^*(v_0)$,
by \cref{lem:interval} (ii), we infer that 
\[
f(v_0)-1\not\in I(v_0).
\]
From \cref{lem:interval}, we also see that the set $I(v_0)$ is allowed. 
It follows immediately that $f(v_0)-2\in I(v_0)$.
By \cref{lem:interval} (i), we find $f(v_0)\not\in J_f^*(v_0)$, a contradiction. 
This completes the proof.
\end{proof}

\end{document}